\documentclass[letterpaper, 10 pt, conference]{ieeeconf}

\usepackage{amsmath}

\usepackage{amsthm}
\usepackage{amsfonts}
\usepackage{dsfont} 
\usepackage{graphicx}
\usepackage{subfigure}
\usepackage{cite}
\usepackage{algorithm}
\usepackage{algpseudocode}
\usepackage{nicefrac}
\usepackage{eufrak}

\newtheorem{defn}{Definition}
\newtheorem{thm}{Theorem}
\newtheorem{lem}{Lemma}
\newtheorem{prop}{Proposition}
\newtheorem{cor}{Corollary}
\newtheorem{note}{Note}
\newtheorem{rem}{Remark}
\newtheorem{exam}{Example}
\newtheorem{assump}{Assumption}
\newtheorem{prob}{Problem}

\newcommand{\e}{\varepsilon}
\newcommand{\G}{\mathcal{G}}
\newcommand{\HH}{\mathcal{H}}
\newcommand{\V}{\mathbb{V}}
\newcommand{\E}{\mathcal{E}}

\newcommand{\R}{\mathbb{R}}
\newcommand{\PP}{\mathbb{P}}

\newcommand{\IM}{\mathrm{Im}}
\newcommand{\Proj}{\mathrm{Proj}}

\usepackage[color=yellow, textsize=small, textwidth=\marginparwidth, draft]{todonotes}   
\IEEEoverridecommandlockouts
\begin{document}

\title{Network Identification: A Passivity and Network Optimization Approach}

\author{Miel Sharf and Daniel Zelazo
\thanks{M. Sharf and D. Zelazo are with the Faculty of Aerospace Engineering, Israel Institute of Technology, Haifa, Israel.
    {\tt\small msharf@tx.technion.ac.il, dzelazo@technion.ac.il}.  This work was supported by the German-Israeli Foundation for Scientific Research and Development.}
}

\maketitle

\begin{abstract}
The theory of network identification, namely identifying the interaction topology among a known number of agents, has been widely developed for linear agents over recent years. However, the theory for nonlinear agents remains less extensive. We use the notion maximal equilibrium-independent passivity (MEIP) and network optimization theory to present a network identification method for nonlinear agents. We do so by introducing a specially designed exogenous input, and exploiting the properties of networked MEIP systems. We then specialize on LTI agents, showing that the method gives a distributed cubic-time algorithm for network reconstruction in that case. We also discuss different methods of choosing the exogenous input, and provide an example on a neural network model. 
\end{abstract}

\section{Introduction}\label{sec.Intro}
Multi-agent systems have been widely studied in recent years, as they present both a variety of applications and a deep theoretical framework. They have been employed across numerous domains, including flocking, formation control, robotics rendezvous, social networks, and distributed estimation \cite{Mesbahi_Egerstedt, Liu2011} \if(0), Sabater2002} \fi. One of the most important aspects in multi-agents systems, both in theory and in practice, is the information-exchange layer, governing which agents interact with each other. Identifying the underlying network of a multi-agent system from measurements is of great importance in many applications. One example is systems biology, in which measurements are used to understand the connection between genes in regulatory networks \cite{BiologicalSystems,Fujita2007}. Another example is international finance, in which past exchange rates between different currencies are used to determine their influence on one another, giving a useful guide for understanding the causal relationship between individual currencies \cite{ExchangeRate}. Other fields with similar problems include social networks \cite{zheleva2012privacy}, physics \cite{\if(0) Boccaletti_Ivanchenko_Latora_Pluchino_Rapisarda,\fi Timme}, neuroscience \cite{Sakkalis,Bressler_Seth}, communication networks \cite{NetworkCoding,FourierMethods} and ecology \cite{\if(0) Monestiez_Bailly_Lagacherie_Voltz,\fi Urban_Keitt}.

The problem of network identification has been widely studied for linear agents. Seminal works dealing with network identification include \cite{Materassi_Innocenti,materassi2009unveiling}, providing exact reconstruction for tree-like graphs, and \cite{Sanandaji_Vincent_Wakin} in which sparse enough topologies can be identified from a small number of observations.  Other important works include \cite{NodeKnockout}, using a node knockout method, and \cite{SieveMethod}, presenting a sieve method for solving the network identification problems for consensus-seeking networks. More recent methods include auto-regressive models \cite{AutoRegressiveModels} and spectral methods \cite{SpectralMethods}. However, a theory for network identification for interacting nonlinear agents is far less developed. We aim to provide in this work a network identification scheme for a wide range of systems, including nonlinear ones.  Our approach relies on a concept widespread in multi-agent systems, namely passivity theory.

Passivity theory is a cornerstone of the theoretical frame work of networks of dynamical systems \cite{Bai_Arcak_Wen}. The main reason is that it allows for the analysis of multi-agent systems to be decoupled into two separate layers, the dynamic system layer and the information exchange layer. Passivity theory was first used to study the convergence properties of network systems in \cite{Arcak}. Many variations and extensions of passivity have been applied in different aspects of multi-agent systems. For example, the related concepts of incremental passivity or relaxed co-coercivity have been used to study various synchronization problems \cite{Stan2007, Scardovi2010}, and more general frameworks including Port-Hamiltonian systems on graphs \cite{Schaft2012}.

One prominent variant is maximal equilibrium-independent passivity (MEIP), which was applied in \cite{SISO_Paper} in order to reinterpret the analysis problem for multi-agent system as a network optimization problem. Network optimization is a branch of optimization theory dealing with optimization of functions defined over graphs \cite{Rockafellar1998}. The main result of \cite{SISO_Paper} showed that the asymptotic behavior of these networked systems is \emph{(inverse) optimal} with respect to a family of network optimization problems. In fact, the steady-state input-output signals of both the dynamical systems and the controllers comprising the networked system can be associated to the optimization variables of either an \emph{optimal flow} or an \emph{optimal potential} problem; these are the two canonical dual network optimization problems described in \cite{Rockafellar1998}. The results of \cite{SISO_Paper} were used in \cite{LCSS_Paper,TAC_Paper} in order to solve the synthesis problem for multi-agent systems.

We aim to use this network optimization framework to provide a network identification scheme for multi-agent systems. We do so by injecting a constant exogenous output, and tracking the output of the agents. By appropriately designing the exogenous input, we are able to differentiate the outputs of the closed-loop system associated to different underlying graphs. The key idea in the proof is that the steady-state outputs are solutions to network optimization problems and they are one-to-one dependent on the exogenous input. 
Our contributions are stated as follows:
\begin{itemize}
\item[i)] We introduce the notion of \emph{indication vectors} for MEIP systems that are used for differentiating the output of networked systems with different underlying graphs.
\item[ii)] We propose various methods for constructing these indication vectors.
\item[iii)] We propose an algorithm exploiting the notion of indication vectors to \emph{solve} a network detection problem. 
\item[iv)] We show that in the case of linear time-invariant (LTI) systems, our solution gives a distributed $O(n^3)$ network detection algorithm, where $n$ is the number of agents.
\end{itemize}
\if(0)
While the results developed in this work are for general systems satisfying the MEIP property, we also treat the special case of linear time-invariant (LTI) systems, allowing us to prove stronger results. Namely, we show that the solution to the network detection problem can be distributed, and its complexity is $O(n^3)$, where $n$ is the number of agents.
\fi
The rest of the paper is organized as follows. Section \ref{sec.background} surveys the relevant parts of the network optimization framework. Section \ref{sec.ProblemFormulation} presents the problem formulation. Section \ref{sec.IndicationVectors} presents the main technical tool used for building the network detection schemes, namely indication vectors, and shows different methods of constructing indication vectors. Section \ref{sec.Applications} uses indication vectors to design a network detection scheme for general MEIP agents. Lastly, we present a case study simulating the network detection methods discussed on a neural network.

\paragraph*{Notations}
We use basic notations from linear algebra. For a linear map $T:U\to V$ between vector spaces, we denote the kernel of $T$ by $\ker{T}$, and the image of $T$ by $\IM(T)$. Furthermore, if $U$ is a subspace of an inner-product space $X$ (e.g., $\mathbb{R}^d$), we denote the orthogonal complement of $U$ by $U^{\perp}$. The notation $A\ge 0$ ($A>0$) means the matrix $A$ is positive semi-definite (positive definite). 
We also use basic notions from algebraic graph theory \cite{Godsil_Royle}. An undirected graph $\mathcal{G}=(\mathbb{V},\mathbb{E})$ consists of a finite set of vertices $\mathbb{V}$ and edges $\mathbb{E} \subset \mathbb{V} \times \mathbb{V}$.  We denote by $k=\{i,j\} \in \mathbb{E}$ the edge that has ends $i$ and $j$ in $\mathbb{V}$. For each edge $k$, we pick an arbitrary orientation and denote $k=(i,j)$. 
The incidence matrix of $\mathcal{G}$, denoted $\mathcal{E}_\mathcal{G}\in\mathbb{R}^{|\mathbb{E}|\times|\mathbb{V}|}$, is defined such that for edge $k=(i,j)\in \mathbb{E}$, $[\mathcal{E_\G}]_{ik} =+1$, $[\mathcal{E_\G}]_{jk} =-1$, and $[\mathcal{E_\G}]_{\ell k} =0$ for $\ell \neq i,j$.

\section{Network Optimization and\\ MEIP Multi-Agent Systems}\label{sec.background}

The role of network optimization theory in cooperative control was introduced in \cite{SISO_Paper}, and was used in \cite{LCSS_Paper,TAC_Paper} to solve the synthesis problem for multi-agent systems.  In this section, we provide an overview of the main results from these works. 

\subsection{The Closed-Loop and Steady-States}
Consider a collection of agents interacting over a network $\mathcal{G}=(\mathbb{V},\mathbb{E})$.  Assign to each node $i\in \mathbb{V}$ (the agents) and each edge $e \in \mathbb{E}$ (the controllers) the dynamical systems,
\begin{align}
	\Sigma_i:
	\left\{\begin{array}{c}
		\dot{x}_i = f_i(x_i,u_i) \\
		y_i = h_i(x_i,u_i)
	\end{array}\right., \,
	\Pi_e: 
	\left\{\begin{array}{c}
		\dot{\eta}_e = \phi_e(\eta_e,\zeta_e) \\
		\mu_e = \psi_e(\eta_e,\zeta_e)
		\end{array} \right. .
\end{align}
 We consider stacked vectors of the form $u=[u_1^T,\ldots,u_{|\mathbb{V}|}^T]^T$ and similarly for $y,\zeta$ and $\mu$ and the operators $\Sigma$ and $\Pi$. The network system is diffusively coupled with the controller input described by $\zeta = \E_\G^Ty$, and the control input to each system by $u = -\E_\G\mu$.  This structure is illustrated in Fig. \ref{ClosedLoop} and we denote the closed-loop system above by the  triple $(\G,\Sigma,\Pi)$.

\begin{figure} [!t] 
    \centering
    \includegraphics[scale=0.4]{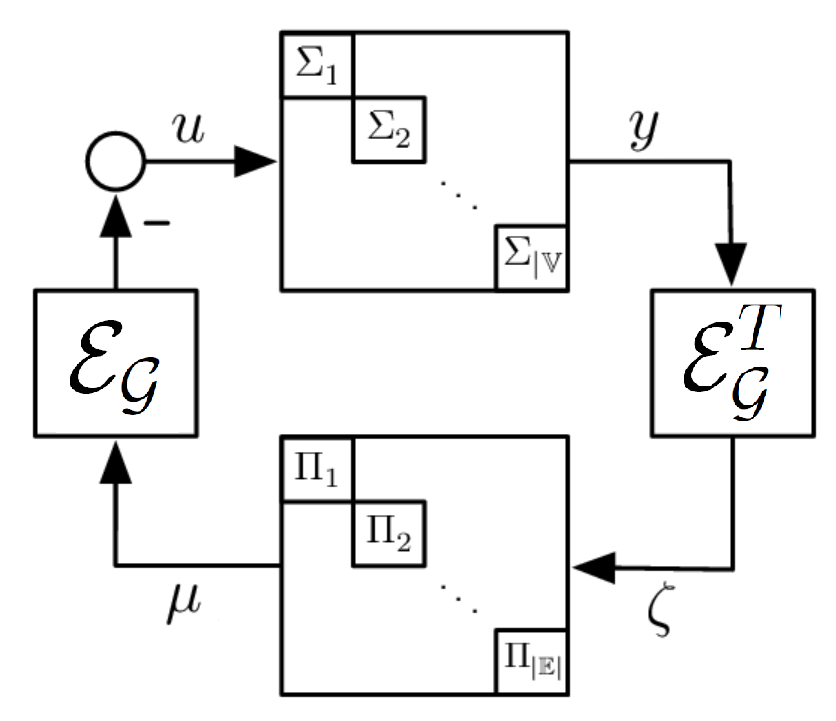}
    \vspace{-10pt}
    \caption{Block-diagram of the closed loop.}
    \label{ClosedLoop}
    \vspace{-15pt}
\end{figure}

Of interest for these systems are the steady-state solutions, if they exist, of the closed-loop. 
Suppose that $(\mathrm{u},\mathrm{y},\zeta,\mu)$ is a steady-state of the system. Then $(\mathrm{u}_i,\mathrm{y}_i)$ is a steady-state input-output pair of the $i$-th agent, and $(\zeta_e,\mu_e)$ is a steady-state pair of the $e$-th edge. This motivates the following definition, originally introduced in \cite{SISO_Paper}.
\begin{defn}
The \emph{steady-state input-output relation} $k$ of a dynamical system is the collection of all steady-state input-output pairs of the system. Given a steady-state input $\mathrm{u}$ and a steady-state $\mathrm{y}$, we define
\begin{align}
k(\mathrm u) = \{\mathrm y:\ (\mathrm{u,y})\in k\} \text{ and }\;
k^{-1}(\mathrm y) = \{\mathrm u:\ (\mathrm{u,y})\in k\}. \nonumber
\end{align}
\end{defn}
Let $k_i$ be the steady-state input-output relation for the $i$-th agent, $\gamma_e$ be the steady-state input-output relation for the $e$-th controller, and $k,\gamma$ be their stacked versions. Then, the network interconnection shown in Fig.\ref{ClosedLoop} imposes on the closed-loop steady-states $(\mathrm{u},\mathrm{y},\zeta,\mu)$ that $\mathrm y\in k(\mathrm u)$, $\zeta=\E_\G^T\mathrm y$, $\mu\in\gamma(\zeta)$, and $\mathrm u=-\E_\G\mu$. Equivalently stated, $\mathrm{y}$ is a steady-state for the system $(\mathcal{G},\Sigma, \Pi)$ if and only if $$0\in k^{-1}(\mathrm y) + \E_\G\gamma(\E_\G^T\mathrm y).$$
The above expression summarizes both the dynamic and algebraic constraints that must be satisfied by the network system to achieve a steady-state solution.

\subsection{MEIP Systems and Convergence of the Closed-Loop}
Convergence of the system $(\mathcal{G},\Sigma,\Pi)$ can be guaranteed under a passivity assumption on the agent and controller dynamics \cite{SISO_Paper}.  
\if(0)
Consider the dynamical system of the form
\begin{align} \label{Upsilon}
\Upsilon : \begin{cases}
\dot{x} = f(x,u) \\
y = h(x,u),
\end{cases}
\end{align}
where $u \in \mathbb{R}$ is the input and $y \in \mathbb{R}$ is the output. In \cite{Hines}, the notion of \emph{equilibrium independent passivity} (EIP) was introduced.  EIP systems requires the existence of a continuous and monotone function that maps constant input signals to constant output signals, i.e., an equilibrium input-output map, and that the system is passive with respect to these input-output pairs.  
An extension of this notion proposed in \cite{SISO_Paper} is \emph{maximal equilibrium independent passivity} (MEIP). For MEIP, we consider the set of all pairs $(u_{ss},y_{ss})$ of steady-state inputs and outputs for $\Upsilon$ and denote this by the \emph{relation} $k_\Upsilon$. Thus, the set of all the steady-state outputs associated with the input $u$ is $k_\Upsilon(u)$, and the set of all steady-state inputs associated with the output $y$ by $k_\Upsilon^{-1}(y)$. These are both set-valued maps, as their image can have more than one point, or no points at all. For example, if $\Upsilon$ is the simple integrator $\dot{x}=u,\: y=x$, then $k_\Upsilon = \{(0,y):\: y\in\mathbb{R}\}$.  This is the main distinguishing point between EIP and MEIP.  In EIP, it is required that the steady-state input-output maps are \emph{functions}, while for MEIP they should be \emph{relations}.  Consequently, the integrator example is not EIP.
\fi 
\begin{defn}[\small{Maximal Equilibrium Independent Passivity \cite{SISO_Paper}}]
Consider the dynamical system of the form
\begin{align} \label{Upsilon}
\Upsilon : \begin{cases}
\dot{x} = f(x,u) \\
y = h(x,u),
\end{cases}
\end{align}
with steady-state input-output relation $r$.
The system $\Upsilon$ is said to be \emph{(output-strictly) maximal equilibrium independent passive} (MEIP) if the following conditions hold:
\begin{enumerate}
\item[i)] The system $\Upsilon$ is (output-strictly) passive with respect to any steady state pair $(\mathrm u,\mathrm y) \in r$.
\item[ii)] The relation $r$ is maximally monotone.  That is, if $(\mathrm u_1,\mathrm y_1),(\mathrm u_2,\mathrm y_2)\in r$ then either $(\mathrm u_1\le \mathrm u_2 \text{\: and\:} \mathrm y_1\le \mathrm y_2)$, or $(\mathrm u_1 \ge \mathrm u_2 \text{\: and\:} \mathrm y_1 \ge \mathrm y_2)$, and $r$ is not contained in any larger monotone relation \cite{Rockafellar_Theorem}.
\end{enumerate}
\end{defn}
Such systems include simple integrators, gradient systems, Hamiltonian systems on graphs, and others (see \cite{SISO_Paper, TAC_Paper} for more examples). We remark that the monotonicity requirement is used to prove existence of a closed-loop steady-state, see \cite{SISO_Paper} or \cite{TAC_Paper} for more details.  

\begin{thm}[\cite{SISO_Paper,LCSS_Paper}]\label{thm.ClosedLoopSteadyStates}
Consider the closed-loop system $(\G,\Sigma,\Pi)$. Assume that the agents $\Sigma_i$ are MEIP, and that the agents $\Pi_e$ are output-strictly MEIP. Then the signals $u,y,\zeta,\mu$ of the closed-loop system converge to some steady-state values ${\mathrm u},{\mathrm y},{\mathrm \zeta},{\mathrm \mu}$ satisfying $0\in k^{-1}({\mathrm y}) + \E_\G \gamma(\E_\G^T {\mathrm y})$.
\end{thm}

\section{Motivation and Problem Formulation} \label{sec.ProblemFormulation}
The problem of network identification we aim to solve can be stated as follows. Given a multi-agent system $(\mathcal{G},\Sigma,\Pi)$, determine the underlying graph structure $\G$ from the network measurements and an appropriately designed exogenous input $w$. Many works on network identification consider networks of consensus-seeking agents \cite{NodeKnockout,SieveMethod},
\begin{align} \label{eq.Consensus}
\dot{x}_i = \sum_{\{i,j\}\in\mathbb{E}} \alpha_{ij}(x_j - x_i) + B_i w_i,
\end{align}
where $w_i$ is the controlled exogenous input for the $i$-th agent, and $\alpha_{ij} = \alpha_{ji}$ are the coupling coefficients. We consider a more general case of (possibly nonlinear) agents interacting over a modified protocol,
\begin{align} \label{eq.System}
\dot{x}_i = f_i(x_i) + \sum_{\{i,j\}\in\mathbb{E}} \alpha_{ij}g_{ij}(h_j(x_j) - h_i(x_i)) + B_i w_i,
\end{align}
where $x_i \in \R$ , and $f_i,g_{ij},h_i:\R\to\R$ are smooth functions. \footnote{The functions $g_{ij}$ are defined for all pairs, even those absent from the underlying graph. It is often assumed in multi-agent systems that each agent knows to run a given protocol (i.e., consensus).}  
Examples of systems governed by \eqref{eq.System}, for appropriate choice of functions $f_i,g_{ij},h_i$, include traffic control models \cite{Traffic}, neural networks \cite{Franci_Scardovi_Chaillet}, and the Kuramoto model for synchronizing oscillators \cite{DORFLER20141539}. We let $f,g,h$ denote the stacked versions of $f_i,g_{ij},h_i$.

In this work, we shall restrict ourselves to the case of $\alpha_{ij} = 1$, i.e., of unweighted graphs \cite{NodeKnockout}. Furthermore, in the model \eqref{eq.Consensus}, the standard assumption is that only certain agents can be controlled using the exogenous input $w_i$ (i.e., $B_i=0$ is possible), and one can observe the outputs of only certain agents. To simplify the presentation, we assume that the exogenous output $w_i$ can be added to all agents, and that the output of all agents can be observed. In that case, we can assume without loss of generality that $B_i = 1$.

We note that the system \eqref{eq.System} is a special case of the closed-loop presented in Fig. \ref{ClosedLoop}, where the agents and the controllers are given by
\begin{align} \label{eq.ClosedLoopStruc}
	\Sigma_i:
	\left\{\hspace{-5pt}\begin{array}{cl}
		\dot{x}_i &\hspace{-5pt}= f_i(x_i)+u_i+w_i \\
		y_i &\hspace{-5pt}= h_i(x_i)
	\end{array}\right.\hspace{-5pt}, 
	\hspace{5pt}\Pi_{ij}: 
	\hspace{-5pt}\begin{array}{cl}
		\zeta_{ij} = g_{ij}(\mu_{ij})
		\end{array} \hspace{-5pt},
\end{align}
and the network is connected using the diffusive coupling $\zeta=\E_\G^Ty$ and $u=-\E_\G \mu$. We would like to use the mechanisms presented in Section \ref{sec.background} to establish network identification results. We make the following assumptions on the agents and controllers, allowing us to use the framework presented in section \ref{sec.background}.  With this model, we will often write the closed-loop as $(\G,\Sigma,g)$.
\begin{assump} \label{Assumption1}
The systems $\Sigma_i$, for all $i \in \mathbb{V}$, are output-strictly MEIP. Furthermore, the controllers  $\Pi_e$, for all $e \in \mathbb{E}$, are MEIP, i.e., $g_{ij}$ are monotone ascending functions.
\end{assump}
\begin{assump} \label{Assumption2}
The inverse of the steady-state input-output relation for each agent, $k_i^{-1}(\mathrm y_i)$, is a smooth function of $\mathrm y_i$. Furthermore, we assume that $g_{ij}(\zeta_{ij})$ is a smooth function of $\zeta_{ij}$, and that the derivative $\frac{dg_{ij}}{d\zeta_{ij}} > 0$ for all $\zeta_{ij} \in \R$.
\end{assump}
Assumption \ref{Assumption2} implies that the integral function $K_i^\star$ associated with $k_i^{-1}$ \cite{SISO_Paper} is smooth and $\nabla K_i^\star = k_i^{-1}$. The assumption on $g_{ij}$ implies that $g_{ij}$ is strictly monotone ascending, and the stronger assumption is made mainly to avoid heavy technical tools.

We will also consider the special case where the agents and controllers are described by linear and time-invariant (LTI) dynamics.  For such systems, the input-output relation $k_i$ for each agent is linear and strictly monotone, and so is the function $g_{ij}$.  When $\Sigma_i$ is an integrator, the input-output relation is given as $\{(0,\mathrm y):\ \mathrm y\in\R\}$. In these cases, $k^{-1}_i$ is a linear function over $\R$. In particular, $k^{-1}_i(\mathrm x_i)=a_i\mathrm x_i$ for some constant $a_i \ge 0$.  We can then define the matrix $A=\mathrm{diag}(a_1,\ldots,a_n)$ such that $k^{-1}(\mathrm x)=A\mathrm x$. Similarly, we denote $g_{ij}(\mathrm x_{ij}) = b_{ij} \mathrm{x}_{ij}$, where $b_{ij} > 0$, and $B=\mathrm{diag}(\cdots,b_{ij},\cdots) >0$. 

We can now formulate two fundamental problems of network detection that we will consider.
\begin{prob} \label{prob.network_differentiation} Consider the network system $(\G, \Sigma, \Pi)$ of the form \eqref{eq.System} satisfying Assumptions \ref{Assumption1} and \ref{Assumption2} with known steady-state input-output relations for the agents and controllers.  Design the control inputs $w_i$ so that it is possible to differentiate the network system $(\G, \Sigma, \Pi)$ from the network system $(\mathcal{H}, \Sigma, \Pi)$, when $\mathcal{H}\ne\G$.
\end{prob}

\begin{prob} \label{prob.network_detection2}
Consider the network system $(\G, \Sigma, \Pi)$ of the form \eqref{eq.System} satisfying Assumptions \ref{Assumption1} and \ref{Assumption2} with known steady-state input-output relations for the agents and controllers, but unknown network structure $\G$.  Design the control inputs $w_i$ such that together with the output measurements of the network, it is possible to reconstruct the graph $\G$. 
\end{prob}

We aim for a solution of both problems, starting with the Problem \ref{prob.network_differentiation} in the sequel. We will later show how to augment the algorithm solving Problem \ref{prob.network_differentiation} to solve the harder problem Problem \ref{prob.network_detection2}.

Note the framework developed in \cite{SISO_Paper} requires constant signals for exogoneous inputs.  Thus, we will consider constant $w_i$, and denote them as $\mathrm w_i$. 
For similar reasons, we can only consider the output measurements of the system in steady-state when reconstructing the graph $\G$. However, in practice we may not be able to wait for the system to converge, and can only know the terminal state up to some approximation error. 
We will deal with this issue by giving a bound on the error one can tolerate. 

\section{Distinguishing Between Different Networks} \label{sec.IndicationVectors}
In this section, we develop the notion of indication vectors used for solving Problem \ref{prob.network_differentiation}, and provide different methods for constructing them. 

\subsection{The Basic Equation and Indication Vectors}
We consider a network system of the form \eqref{eq.ClosedLoopStruc}. We first study constant exogenous input vectors $\mathrm{w}$ that can differentiate between two different network systems $(\G,\Sigma,\Pi)$ and $(\mathcal{H},\Sigma,\Pi)$.  In this direction, we provide a result relating the constant exogenous inputs $\mathrm{w}$ to the network steady-states.
\begin{prop} \label{prop.Indication}
Under Assumptions \ref{Assumption1} and \ref{Assumption2}, for any $\mathrm w\in\mathbb{R}^{n}$, the vector $\mathrm{y}\in\mathbb{R}^n$ is a steady-state of the closed loop system $(\G,\Sigma,\Pi)$ if and only if
\begin{align}\label{eq.SteadyState}
k^{-1}(\mathrm{y})+\E_\G g\left(\E_\G^T\mathrm{y}\right)= -\mathrm w.
\end{align}
\end{prop}
\begin{proof}
The result follows directly from Theorem \ref{thm.ClosedLoopSteadyStates} using $\gamma(\mathrm \zeta)=g(\mathrm \zeta$), $k^{-1}(\mathrm y)=\mathrm u+\mathrm w$ and the network connection $\zeta=\E^Ty,\ u=-\E\mu$. We note that this is an equality and not inclusion due to Assumption \ref{Assumption2}. 
\end{proof}
We denote a solution of \eqref{eq.SteadyState} as $\mathrm{y}_\G$. Furthermore, we note that \eqref{eq.SteadyState} is graph dependent, and that the steady-state output $\mathrm{y}_\G$ can be measured from the network (as the network converges to a steady-state by Theorem \ref{thm.ClosedLoopSteadyStates}). The idea now is to choose the bias vectors $\mathrm w$ wisely so that different graphs will have different terminal outputs. 
\begin{defn}\label{def.IndicationVectors}
Consider a closed-loop system of the form \eqref{eq.System} satisfying Assumptions \ref{Assumption1} and \ref{Assumption2}, where the controllers $g_{ij}$ have been determined for all possible pairs $\{i,j\}$. 
Let $\mathfrak{G}$ be a collection of graphs over $n$ vertices. A vector $\mathrm w\in\R^n$ is called a \emph{$\mathfrak{G}$-indication vector} if for any two graphs $\G,\mathcal{H} \in \mathfrak{G}$ with $\G \neq \mathcal{H}$, the steady-state output of $(\G,\Sigma,g)$ is different from the steady-state output of $(\mathcal{H},\Sigma,g)$. 
In other words, one has that $\mathrm{y}_{\G} \neq \mathrm{y}_{\mathcal{H}}$.
\end{defn}

Assume for now that the agents and controllers are LTI. We can now restate \eqref{eq.SteadyState} in a new manner, involving linear inequalities. In turn, this will manifest in stronger results later. The following result will be useful in the analysis.

\begin{prop}\label{prop.LTIinvert}
If $A \ne 0$, then for any connected graph $\G$, the matrix $S=A+\E_\G B \E_\G^T$ is invertible.
\end{prop}
\begin{proof}
Note that $\E_\G B \E_\G^T \ge 0$ and that $A\ge 0$, implying that $S \ge 0$. Furthermore, the kernel of $\E_\G B \E_\G^T$ consists solely of the span of the all-ones vector, $\mathds{1}$.  It follows that $\mathds{1}^TA\mathds{1} = \sum_{i=1}^{|\V|} a_i >0$, completing the proof. 
\end{proof}
Proposition \ref{prop.LTIinvert} allows an explicit form for \eqref{eq.SteadyState} for the case $A\ne 0$ by inverting the matrix in question,
$$\mathrm{y}_\G = -(A+\E_\G B\E_\G^T)^{-1}\mathrm w = -X_{\G,A\neq 0} \mathrm w.$$ 
If $A=0$, however, we note that for any $\G\in\mathfrak{G}$, the linear operator $\E_\G B\E_\G^T$ preserves $\IM(\E_\G) = \mathds{1}^\perp$, and moreover, it is invertible when restricted to it. Thus, we denote the restriction of $\E_\G B\E_\G^T$ on $\mathds{1}^\perp$ by $Y_\G$, and obtain 
$$\mathrm{y}_\G  = -Y_\G \Proj_{{\mathds{1}^\perp}} \mathrm w = -X_{\G,A=0} \mathrm w.$$ 

These linear relations between the steady-state output $\mathrm y_\G$ and the constant exogenous input $\mathrm w$ allows for an easier statement of the definition of indication vectors. 
\begin{prop}\label{cor.LTI_Indication}
For LTI agents and controllers, and for a vector $\mathrm w \in \R^n$, the following statements hold:
\begin{itemize}
\item[i)] If $A \ne 0$, $\mathrm w$ is a $\mathfrak{G}$-indication vector if and only if for any two different graphs $\G,\mathcal{H}\in\mathfrak{G}$, we have $X_{\G,A\neq 0}\mathrm{w} \neq X_{\mathcal{H},A\neq 0} \mathrm{w}$.
\item[ii)] If $A = 0$, $\mathrm w$ is a $\mathfrak{G}$-indication vector if and only if any two different graphs $\G,\mathcal{H}\in\mathfrak{G}$, we have $X_{\G,A=0}\mathrm{w} \neq X_{\mathcal{H},A=0} \mathrm{w}$.
\end{itemize}
\end{prop}
We will use $X_\G$ for notational simplicity. 
We first note the following interesting property of $X_\G$.
\begin{prop}\label{rem.uniquenessofX}
If $\G \neq \mathcal{H}$ then $X_{\G} \neq X_{\mathcal{H}}$. 
\end{prop}
\begin{proof}
Suppose first that $A\neq 0$. We can reconstruct the weighted graph Laplacian $\E_\G B \E_\G^T$ from $X_\G$ using the relation $\E_\G B\E_\G^T = - A + X_\G^{-1}$, thus $X_{\G} = X_{\mathcal{H}}$ implies $\G = \mathcal{H}$, as $B > 0$. If $A=0$, we note that $Y_\G = -X_\G^{-1}$ on the set $\mathds{1}^\perp$. This determines the graph Laplacian, as it is the projection of the weighted graph Laplacian on $\mathds{1}^\perp = \ker(\E_\G B\E_\G^T)^\perp$. 
\end{proof}

After restating the definition of indication vectors for LTI systems, we return to the case of general agents and controllers satisfying Assumptions \ref{Assumption1} and \ref{Assumption2}. Given an indication vector $\mathrm w$, we can quantify how much it can differentiate between different graphs. We do so with the following definition.
\begin{defn} \label{def.epsilon}
The \emph{separation index} of $\mathrm{w}$, denoted $\e=\e(\mathrm{w})$, is defined as the minimal distance between $\mathrm{y}_{\G}$ and $\mathrm{y}_{\mathcal{H}}$ where $\G\ne \mathcal{H}$, i.e., $\e = \min_{\mathcal{G}\ne \mathcal{H}} \|\mathrm{y}_{\G}-\mathrm{y}_{\mathcal{H}}\|$, where the minimization is over all graphs in $\mathfrak{G}$.
\end{defn}

\begin{rem} \label{rem.epsilon}
The separation index $\e$ acts as a bound on the error we can tolerate when computing the steady-state output of the closed-loop system. This error can be comprised of both numerical errors, as well as errors arising from early termination of the system (i.e., before reaching steady-state). Indeed, suppose we want to differentiate between $\G,\mathcal{H}$. We know that $\|\mathrm{y}_{\G} - \mathrm{y}_{\mathcal{H}}\| \ge \e$. Suppose we have the terminal output $\mathrm{y}$ of the closed-loop system for either $\G$ or $\mathcal{H}$.  By the triangle inequality, if $\|\mathrm{y-y}_{\G}\| < 0.5\e$ then $\|\mathrm{y-y}_{\mathcal{H}}\| \ge 0.5\e$ and vice versa. Thus, if we know that the sum of the numerical and early termination errors is less than $0.5\e$, we can correctly choose the underlying graph by choosing which of $\mathrm y_{\G}$ and $\mathrm y_{\mathcal{H}}$ is closer to $\mathrm y$.
\end{rem}

\begin{rem} \label{rem.epsilonForLinearSystems}
Consider the case of LTI agents and controllers. In that case, Proposition \ref{cor.LTI_Indication} implies that the relation between $\mathrm{y}_\G$ and $\mathrm w$ is linear. Thus, for any constant $\beta>0$, the separation index satisfies $\e(\beta\mathrm w) = \beta\e(\mathrm w)$.
\end{rem}

\subsection{Methods of Choice for Indication Vectors}
We now explore how to construct indication vectors for a multi-agent system of the form \eqref{eq.System} satisfying Assumptions \ref{Assumption1} and \ref{Assumption2}.  In this sub-section, we present several methods for doing so.
\subsubsection*{Randomization}
Our first approach is to construct the indication vectors via randomization. We claim that random vectors $\mathrm w\in\mathbb{R}^{n}$ are indication vectors with probability $1$.
\begin{thm} \label{Thm.RandMethod}
Let $\PP$ be any continuous probability distribution on $\R^{n}$, and let $\mathfrak{G}$ be any collection of graphs over $n$ nodes. Then $\PP(\mathrm w \text{ is a $\mathfrak{G}$-indication vector}) = 1$.
\end{thm}
\begin{proof}
Recall that $\mathrm w$ is not an indication vector if and only if there are two graphs $\G_1,\G_2$ and a vector $\mathrm{\mathrm y}\in\mathbb{R}^{|\V|}$ such that
\begin{align*}
k^{-1}(\mathrm{y})+\E_{\G_i}g(\E_{\G_i}^T\mathrm{y}) = -\mathrm w,\; i=1,2.
\end{align*} 
Subtracting each equation from the other gives \begin{align} \label{eq.RandomForbiddenSurface}
\E_{\G_1}g(\E_{\G_1}^T\mathrm y) - \E_{\G_2}g(\E_{\G_2}^T\mathrm y) = 0.
\end{align}
For each $\G_1,\G_2$, the collection of solutions to \eqref{eq.RandomForbiddenSurface} forms a set, and note that $\mathrm w$ is an indication vector if and only if the solutions $\mathrm{y}$ are not in any of these sets. Define 
$$ F(y) = \E_{\G_1}g(\E_{\G_1}^T\mathrm y) - \E_{\G_2}g(\E_{\G_2}^T\mathrm y), $$
so that $F:\R^n\to\R^n$ is a smooth function. Its differential is given by
$$ \nabla F(y) = \E_{\G_1}\nabla g(\E_{\G_1}^T\mathrm y)\E_{\G_1}^T - \E_{\G_2}\nabla g(\E_{\G_2}^T\mathrm y)\E_{\G_2}^T, $$
where $\nabla g= \mathrm{diag}(\frac{dg_{ij}}{d\zeta_{ij}})$ is the derivative of $g$. Because $\frac{dg_{ij}}{d\zeta_{ij}} > 0$ by Assumption \ref{Assumption2}, $\nabla F(y)$ is the difference of two weighted graph Laplacians, with underlying graphs $\G_1,\G_2$ and positive weights. Thus $\nabla F$ never vanishes, and Lemma \ref{lem.Implicit} implies that the solutions of \eqref{eq.RandomForbiddenSurface} form a zero measure set. Thus $\mathrm w$ is an indication vector if and only if the solutions $\mathrm{y}$ are not in the finite union of the zero measure sets defined by \eqref{eq.RandomForbiddenSurface}, i.e., a zero-measure set.

The mapping between $-\mathrm w$ and $\mathrm{y}$,  $-\mathrm w = k^{-1}(\mathrm{y}) + \E_\G g(\E_\G^T\mathrm{y})=G(\mathrm{y})$, is smooth and strictly monotone, meaning that it is the gradient of a strictly convex and smooth function. Thus, the inverse function $\mathrm{y}=G^{-1}(\mathrm w)$ is a smooth and strictly convex function, as the gradient of the dual function, which is also strictly convex and smooth \cite{Rockafellar1998}. This implies that $G^{-1}$ is absolutely continuous \cite{AbsoluteContinuity}, sending the zero measure sets to zero measure sets. In turn, the set that $\mathrm y$ has to avoid (for $\mathrm w$ to be an indication vector) is zero-measure, meaning that the corresponding set that $\mathrm w$ has to avoid is also zero measure. But $\PP$ is a continuous probability measure, and thus the probability of the zero-measure set that $\mathrm{w}$ has to avoid is zero. This completes the proof.
\end{proof}

This method works under the Assumptions \ref{Assumption1} and \ref{Assumption2}, but can produce stronger results when considering LTI agents and controllers.  In particular, we can estimate the separation index of a randomly chosen vector.
\begin{cor} \label{thm.RandEpsilon}
Suppose the agents and controllers are LTI. Furthermore, suppose that $\mathrm w$ is sampled according to the standard Gaussian probability measure $\PP$ on $\R^n$. Define $\beta=\min\{a_1,...,a_n\}$ if $A\ne 0$, and $\beta = \min\{b_{ij}\}/{n \choose 2}$ otherwise.
 Then for any $\delta>0$, the separation index $\e=\e(\mathrm{w})$ satisfies $\delta \le \e$ with probability $\ge 1-2^{n^2}(2\Phi(\delta/2\beta)-1)$, where $\Phi$ is the cumulative distribution function of a standard Gaussian random variable.
\end{cor}
The proof is available in the appendix.

\begin{rem}
Corollary \ref{thm.RandEpsilon} and Remark \ref{rem.epsilonForLinearSystems} give a viable method for assuring that the distance between different terminal states of the system (corresponding to different base graphs) is as large as desired. First, choose a desired degree of security $p$, which is the probability of the choice to be successful (say $p=.9999$). Choose $\delta$ so that $p \le 1-2^{n^2}(2\Phi(\delta/2\beta)-1)$. Now choose $\mathrm w$ randomly according to a standard Gaussian distribution, and multiply it by $1/\delta$.
\end{rem}

Let us present another, more constructive approach for designing indication vectors, building upon number theory. 
\subsubsection*{Bases of Computation}
For the rest of this subsection, we continue with LTI agents. We can apply this method if the elements of $A$ are rational. In this case, the elements of $X_\G$ are all rational. The idea is that we can reconstruct the entries of $X_\G$ from $X_\G \mathrm w$ if $\mathrm{w}$ is of the form $\mathrm w=[1,M,M^2,...,M^{n-1}]^T$ for $M$ large enough.
\begin{exam}
Suppose that $C=[a,b,c]$ is a vector with positive integer entries having a numerator no greater than $9$. Take $\mathrm w=[1,10,100]^T$. Then $C\mathrm w=a+10b+100c$ is a three-digit number, and we can reconstruct $C$ by looking at the three digits individually - $a$ is the unit digit, $b$ is the tens digit, and $c$ is the hundreds digit.
\end{exam}
We can generalize this to a more general framework.
\begin{thm} \label{thm.BasesOfComputation}
Suppose that $A,B$ are rational, the denominators of all entries of the matrices $\{X_\G\}_{\G\in\mathfrak{G}}$ divide $D$, and that the numerator (in absolute value) is no larger than $N$. Let $M$ be any integer larger than $(2N+1)D$. Then the vector ${\mathrm w}=[1,M,...,M^{n-1}]^T$ is a $\mathfrak{G}$-indication vector.
\end{thm}
\begin{proof}
Each element in the product $X_\G \mathrm w$ corresponds to a single row of $X_\G$ multiplied with $\mathrm w$, so it's enough to reconstruct a row. We take a single row of $X_\G$ and mark it as $[\frac{p_1}{q_1},\cdots,\frac{p_n}{q_n}]^T$, where $|p_i|\le N$ and $q_i$ divides $D$. We let $\mathrm{R} = (X_\G \mathrm w)_i$. Therefore, 
\[
\mathrm{R}=\begin{bmatrix}\frac{p_1}{q_1}&\cdots & \frac{p_n}{q_n}\end{bmatrix} \mathrm w = \frac{p_1}{q_1} + \frac{p_1}{q_1}M+\cdots+\frac{p_n}{q_n}M^{n-1}.
\]
We can define $m_i = \frac{D}{q_1}$, which is an integer no larger than $D$, and multiply both sides of the equation by $D$ to obtain
\[
D\,\mathrm{R} = m_1p_1 + m_2p_2M+\cdots+m_np_nM^{n-1}.
\]
Note that $m_ip_i$ is an integer lying between $-ND$ and $ND$. We can add $\sum_{i=0}^{n-1} (NDM^i)$ to both sides of the equation, leading to \small
\[
D\mathrm{R} + \sum_{i=0}^{n-1} (NDM^i) = (m_1p_1+ND) +\cdots+(m_np_n+ND) M^{n-1}.
\]\normalsize
The left hand side is known, and the coefficients in the right hand side are integers between $0$ and $2ND$. Thus, writing $D\mathrm{R} + \sum_{i=0}^{n-1} (NDM^i)$ in the system with radix $M$, the numbers $m_ip_i+ND$ can be computed by looking at the individual digits. Deducting $ND$ and dividing by $D$ gives all of the entries $\frac{p_i}{q_i}$, allowing reconstruction.
\end{proof}

\section{Indication Vectors For Network Detection} \label{sec.Applications}
Up to now, we have dealt with indication vectors, which give an easy way of solving Problem \ref{prob.network_differentiation}, i.e., differentiating between closed-loop systems of the form \eqref{eq.System} which differ only on underlying graph level. We claim that we can go a step further and solve Problem \ref{prob.network_detection2}, i.e., reconstructing the underlying graph of a system of the form \eqref{eq.System}.
We now present the network reconstruction scheme based on indication vectors. The key notion that will allow us to take the leap is through the use of appropriate {look-up tables}.

Look-up tables are tables comprising of two columns, one called the key and the other called the value, that act like oracles and are designed to decrease runtime computations \cite{}. The key is usually easy to come by, and the value is usually harder to find. Examples of look-up tables include mathematical tables, like logarithm tables and sine tables. Other examples include phone books and other databases like hospital or police records. 

We now state the main result regarding network detection for general agents and controllers, focusing on the LTI case later. 
\begin{prop} \label{prop.NetworkDetection}
Let $(\G,\Sigma,g)$ be a network system of the form \eqref{eq.System} satisfying Assumptions \ref{Assumption1} and \ref{Assumption2}. Then for any indication vector $\mathrm{w}$, there exists an algoritgm solving Problem \ref{prob.network_detection2} using only a single exogenous output, namely $\mathrm{w}$.
\end{prop}
\begin{proof}
Let $\mathfrak{G}$ be the collection of all graphs on $n$ vertices. We construct a $\mathfrak{G}$-indication vector $\mathrm w$ using Theorem \ref{Thm.RandMethod}. Before running the system, we build a lookup table with keys being graphs $\mathcal{H}\in\mathfrak{G}$, and values being the outputs $\mathrm y_{\mathcal{H}}$, which can be computed by \eqref{eq.SteadyState}. Now, run the closed-loop system with the input $\mathrm w$.
By definition of a $\mathfrak{G}$-indication vector, we know that the terminal output $\mathrm y$ of closed-loop system completely classifies the underlying graph $\G$, i.e., different underlying graphs give rise to different terminal outputs. We can now reconstruct the graph $\G$ by comparing $\mathrm y$ to the values of the look-up table, finding the graph $\mathcal{H}$ minimizing $\|\mathrm y - \mathrm y_\mathcal{H}\|$. Then because $\mathrm w$ can differentiate the systems $(\G,\Sigma,g)$ and $(\mathcal{H},\Sigma,g)$ if $\G \neq \mathcal{H}$, we must have that $\G=\mathcal{H}$. 
\end{proof}

\begin{rem}
In the proof above, we assumed that the closed-loop system is run until the output converges. However, in practice, both numerical errors and early termination errors give us a skewed value of the true terminal output of the closed-loop system, as was discussed in Remark \ref{rem.epsilon}. In the algorithm presented above, we can tolerate an error of up to $0.5\e(\mathrm w)$ in the value of $\mathrm y$.
\end{rem}

\begin{rem}
We should note that in order to implement the network detection scheme in the proof of Proposition \ref{prop.NetworkDetection}, we need an observer with access to the look-up table, the output of all of the agents, and the input $\mathrm{w}$. This network detection scheme is not distributed in the sense that it requires one observer to know the outputs of all of the agents. The size of the look-up table increases rapidly with the number of nodes if we don't assume anything about the underlying graph. One should note that should recall that the computation can be done offline, and that it can be completely parallelized - we are just comparing the entries of the table to the measured output. Furthermore, if we add additional assumptions on the graph (e.g., the underlying graph is a subgraph of some known graph), the size of the look-up table drops significantly.
\end{rem}

We can prove a stronger result for the case of LTI agents and controllers, namely a distributed implementation strategy and an analysis of the algorithm complexity. 
\begin{thm} \label{thm.NetworkDetectionLTI}
Let $(\G,\Sigma,g)$ be a network system of the form \eqref{eq.System} satisfying Assumptions \ref{Assumption1} and \ref{Assumption2}, consisting of LTI agents and controllers, and that the matrices $A,B$ have rational entries. Then there exists a distributed $O(n^3)$ algorithm solving Problem \ref{prob.network_detection2}. It requires to run the system only once, with a specific constant exogenous input $\mathrm w$.
\end{thm}

\begin{rem}
In the case of LTI agents and controllers, finding the graph $\G$ is roughly equivalent to finding $X_\G$. The distributive nature of the algorithm is manifested in the fact that the $i$-th row of $X_\G$ is computed solely from the terminal output of the $i$-th agent.
\end{rem}
\begin{proof}
We pick an indication vector using the method of Theorem \ref{thm.BasesOfComputation}. The proof of Theorem \ref{thm.BasesOfComputation} gives an easy way to reconstruct $X_\G$'s $i$-th row from the output of the $i$-th agent, taking $O(n)$ time. Doing this for all agents takes $O(n^2)$ times, and gives us $X_\G$. Afterward, we can reconstruct the graph Laplacian using the formula $\E_\G B\E_\G^T = -A - X_\G^{-1}$ in $O(n^3)$ time, and then find the underlying grpah by looking at the non-zero off-diagonal entries of it. This completes the proof.
\end{proof}

\begin{note}
In the LTI case, we do not use look-up tables, but give a different solution relying on bases of computation. This allows us to have only a polynomial increase in time, and exempts us from worrying about storage issues. 
\end{note}

\section{Case Study : Neural Network}
We consider a continuous neural network, as appearing in \cite{NeuralNetworkExample}, on $n$ neurons of one species. The governing ODE has the form,
\begin{align}\label{neural}
\dot{V_i} = -\frac{1}{\tau_i}V_i + b_i \sum_{j\sim i} (\tanh(V_j)-\tanh(V_i)) + \mathrm{w}_i
\end{align}
where $V_i$ is the voltage on the $i$-th neuron, $\tau_i > 0$ is the self-correlation time of the neurons, $b_i$ is a coupling coefficient, and the external input $w_i$ is any other input current to neuron $i$. We run the system with $10$ neuron. The correlation times were chosen randomly between $0.5_{sec}$ and $1_{sec}$. The homogeneity requirement on the network forces us to take equal $b_i$-s over all agents, and we choose $b_i = 0.1$. We should note that one can show that the agents, modeled by $\dot{x_i} = -\frac{1}{\tau_i}x_i + u_i ; y_i = \tanh(x_i)$, are output striclty MEIP, namely the following storage function can be used for to prove output-strict passivity with respect to $(\mathrm u_i,\tanh(\tau_i\mathrm u_i))$,
$ V_i (x_i) = \int_0^{x_i} \tanh(s)ds - \int_0^{\tau_i \mathrm u_i} \tanh(s)ds - \tanh(\tau_i \mathrm{u}_i)(x - \tau_i \mathrm u_i). $

We choose an indication vector as in the proof of Proposition \ref{prop.NetworkDetection}, and run the system with the original underlying graph, showing in Fig. \ref{Graph}. The output of the system can be seen in Fig. \ref{neural_net_traj}. We first run the system for 10 seconds (enough for convergence). After 10 seconds, the red edge in Fig. \ref{Graph} gets cut off.  We can see that the output of agent \#6 (in light blue) and agent \#10 (in yellow) change meaningfully, so we are able to detect the change in the underlying graph. After ten more seconds, another edge gets removed from the graph, this time the blue one. We can see that again the outputs of two agents, \#1 (in black) and \#2 (in pink), are changed by a measurable amount, allowing to detect the second change in the underlying graph, which is now unconnected. Finally, after a total of 20 seconds, we reintroduce both of the removed edges. We can see that the system has returned to its original steady-state. 

\begin{figure}[!h] 
   \centering
   \includegraphics[scale=0.35]{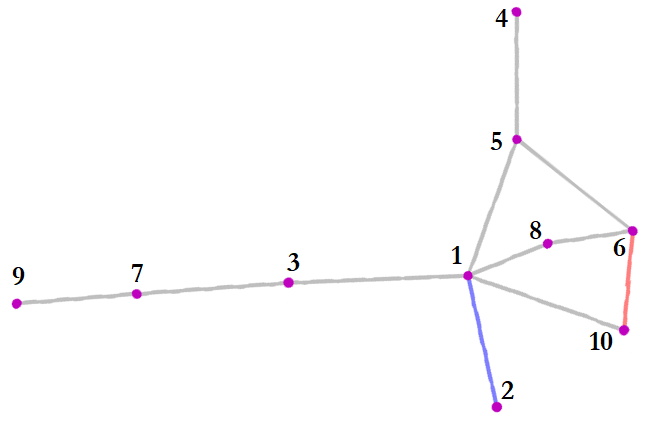}
    \caption{The interaction graphs simulated in the case study. The red edge is cut after 10 seconds, and the blue edge is cut after 20 seconds.}
    \label{Graph}
\end{figure}

\begin{figure} [!h] 
    \centering
    \includegraphics[scale=0.38]{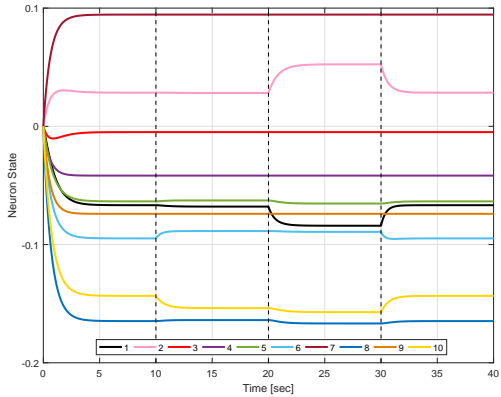}
    \caption{Trajectories of the neural network \eqref{neural} with changes in the underlying network.}
    \label{neural_net_traj}
   \vspace{-15pt}
\end{figure}

\section{Conclusion}
In this work we presented a procedure operating on a network system that allows for the reconstruction of the underlying network with no prior knowledge on it, but only on the agents. This was done through the novel notion of indication vectors, that were achieved for general maximally equilibrium-independent passive agents, allowing for detection of the underlying network in a very general case. We have found stronger results for LTI agents, allowing a distributed cubic-time reconstruction of the underlying network, while dealing with numerical errors present in the system. We have exhibited the use of these indication vectors in a network detection algorithm, and demonstrated the results in a simulation.

\bibliographystyle{ieeetr}
\bibliography{main}

\appendix
This appendix is dedicated to the proof of Corollary \ref{thm.RandEpsilon}, as well as a lemma for the proof of Theorem \ref{Thm.RandMethod}. We start with the latter.
\begin{lem} \label{lem.Implicit}
Let $F:\R^n \to \R^m$ be a smooth function, and consider the set $Z=\{x:\ F(x) = 0\}$. Suppose that the differential $\nabla F(x)$ is not the zero matrix at any point $x\in Z$. Then $Z$ is a zero-measure set. 
\end{lem}
\begin{proof}
We denote the coordinates of $F$ by $F=(F_1,...,F_m)$. We also let $Z_j = \{x:\ F_j(x) = 0\}$ for $j=1,\ldots,m$, so that $Z \subseteq Z_j$. Let $x=(x_1,\ldots,x_n)\in Z$. We know that $\nabla F(x) \neq 0$, so there's some $j$ such that $\nabla F_j(x) \neq 0$ in a small neighborhood of $x$. Thus, by the implicit function theorem, we can express one coordinate as a smooth function of the others. For ease of writing, assume without loss of generality that its the $n$-th coordinate, and let $\phi$ be the smooth function such that near $x$, the set $Z_j$ is given by $y_n = \phi(y_1,...,y_{n-1})$. 

More precisely, we can find a small cube $Q$ containing $x$ such that  the set $Z_j\cap Q$ is given by $y_n = \phi(y_1,...,y_{n-1})$. If we attempt to compute the volume of $Z_j\cap Q$, we can do so using the integral $\int_Q 1_{y_n = \phi(y_1,...,y_{n-1})}(y)dx$, where $1_E$ is the indicator function of $E$. This integral is obviously zero due to Fubini's theorem, and the fact that given $y_1,\ldots,y_{n-1}$, only one $y_n$ can satisfy the equation $y_n = \phi(y_1,...,y_{n-1})$. Thus the volume of $Z\cap Q$, which is smaller than the volume of $Z_j \cap Q$ as $Z_j \subset Z$, nulls.

Up to now, we showed that if $x\in Z$ the there exists some small open set $Q=Q_x$ such that $Z\cap Q_x$ is of measure zero. If we let $Ball_r$ denote the closed ball of radius $r$ around the origin, then $Z\cap B_r$ is compact, and $\{Q_x :\ x\in Z\cap B_r\}$ is an open cover. Thus $Z\cap B_r$ is contained in the union of finitely many sets of the form $Z\cap Q_x$, each of them having zero measure, implying that $Z\cap B_r$ is zero measure. We note that $Z$ is the countable union of the sets $Z\cap B_1, Z\cap B_2, Z\cap B_3,...$, meaning that it is the countable union of zero measure sets, hence a zero measure set itself.
\end{proof}

We now shift focus toward the proof of Theorem \ref{thm.RandEpsilon}. We first prove a lemma:
\begin{lem}
For any connected graphs $\G,\HH$, $\bar\sigma(X_\G-X_\HH) \le C_{\G,A}+C_{\HH,A}$ where we define
\begin{align*}
C_{\G,A} = \begin{cases} \frac{1}{\underline{\sigma}(A+\E_\G B\E_\G^T)} & A\neq 0 \\
\frac{1}{\underline{\sigma}(Y_\G)} & A=0 \end{cases}.
\end{align*}
\end{lem}

\begin{proof}
First, $\bar{\sigma}(X_\G-X_\HH) \le \bar{\sigma}(X_\G)+\bar{\sigma}(X_\HH)$ as $\bar\sigma(\cdot)$ is a norm on the space of matrices. The proof is now completed using the formula $\bar\sigma(C^{-1})=\frac{1}{\underline\sigma(C)}$.
\end{proof}

We can now prove Theorem \ref{thm.RandEpsilon}.

\begin{proof}
The distance between the terminal state associated with $\G$ and the one associated with $\HH$ is $\|(X_\G-X_\HH)^2\|$, where $||\cdot||$ is the standard Euclidean norm. We fix some $\G,\HH$ and let $F=X_\G-X_\HH$. We use the SVD decomposition to write $F^TF = U^TDU$ where $U$ is an orthogonal matrix and $D=\mathrm{diag}(\sigma_1^2,...,\sigma_{n}^2)$ is a diagonal matrix entries being the singular values of $F$. Using the fact that $\mathrm w$ and $U^{-1}\mathrm w$ both distribute according to $\PP$, we see that: \small 
\begin{align*}
&\PP(||F\mathrm w||>\delta)=\PP(||F\mathrm w||^2 > \delta^2) =\PP(||FU^T\mathrm w||^2 > \delta^2)=\\
&\PP(\mathrm w^TUF^TFU^T\mathrm w > \delta^2) = \PP(\mathrm w^TD\mathrm w > \delta^2) =\PP(\sum_{i=1}^{|n|} \sigma_i^2\mathrm w_i^2 > \delta^2).
\end{align*}  \normalsize
Now, we note that the entries $\mathrm w_i$ of $\mathrm w$ are all standard Gaussian random variables, and that they are independent. Thus, we can estimate: \small
\begin{align*}
\PP(\sum_{i=1}^n \sigma_i^2\mathrm w_i^2 > \delta^2)\ge \PP(|\mathrm w_1| \ge \frac{\delta}{\sigma_1}) = 2-2\Phi\left(\frac{\delta}{\bar\sigma(F)}\right)
\end{align*} \normalsize
Thus, we know that for any pair of graphs $\G,\HH$, the chance that $||X_\G \mathrm w - X_\HH \mathrm w||> \delta$ is at least $2-2\Phi(\frac{\delta}{\bar\sigma(X_\G-X_\HH)})$.

Now, we use the lemma and the fact that $\Phi$ is monotone increasing to bound the probability that  $||X_\G \mathrm w - X_\HH \mathrm w||> \delta$ from below by $2-2\Phi(\frac{\delta}{C_{\G,A}+C_{\HH,A}})$. We bound each $C_{\G,A}$ from above. First, if $A\neq 0$ then $\underline{\sigma}(A+\E_\G B\E_\G^T) \ge a$, and otherwise $\underline{\sigma}(Y_\G) = \min\{b_{ij}\}\lambda_2(\G) \ge \min\{b_{ij}\}{n \choose 2}^{-1}$ \cite{Rad_Jalili_Hasler}. Because there are a total of $2^{n \choose 2}$ possible graphs on $n$ nodes, thus a total of $(2^{n \choose 2})^2\le 2^{n^2}$ of pairs $\G,\HH$ to consider, and by the union bound, the chance that $\e<\delta$ is no more than \small
\begin{align*}
\sum_{\G,\HH}(1-\Phi(-\frac{\delta}{\bar\sigma(X_\G-X_\HH)})) \le 2^{n^2}\bigg(2\Phi\bigg(\frac{\delta}{2\beta}\bigg)-1\bigg).
\end{align*} \normalsize
\end{proof} 

\end{document}